\documentclass[letterpaper,12pt]{amsart}
\usepackage{amsmath,amssymb,amsxtra,amsthm, amstext, amscd,amsfonts,fancyhdr, hyperref,enumerate}
\usepackage[OT2,T1]{fontenc}
\DeclareSymbolFont{cyrletters}{OT2}{wncyr}{m}{n}
\DeclareMathSymbol{\Sha}{\mathalpha}{cyrletters}{"58}

\newcommand{\bC}{{\mathbb{C}}}

\newcommand{\bQ}{{\mathbb{Q}}}
\newcommand{\bR}{{\mathbb{R}}}

\newcommand{\bZ}{{\mathbb{Z}}}

  \newcommand{\N}{{\mathcal{N}}}

  \newcommand{\R}{{\mathcal{R}}}
\renewcommand{\S}{{\mathcal{S}}}

  \newcommand{\HH}{\mathcal{H}}

\newcommand{\Rep}{\operatorname{Rep}}

\newcommand{\GL}{\operatorname{GL}}

\newcommand{\Aut}{\operatorname{Aut}}

\newcommand{\DD}{\mathbf{D}}

\newcommand{\CC}{\mathbf{C}}

\newcommand{\ep}{\varepsilon}

\newcommand{\upchi}{{\raise.35ex\hbox{$\chi$}}}

\makeatletter
\@namedef{subjclassname@2010}{%
  \textup{2010} Mathematics Subject Classification}
\makeatother

\newtheorem{theorem}{Theorem}[section]

\newtheorem{lemma}[theorem]{Lemma}
\newtheorem{conjecture}[theorem]{Conjecture}

\theoremstyle{definition}

\numberwithin{equation}{section}

\begin{document}

\title[Representation of $k$-free integers by binary forms]{On the representation of $k$-free integers by binary forms}

\author{C.~L.~ Stewart}
\address{Department of Pure Mathematics \\
University of Waterloo \\
Waterloo, ON\\  N2L 3G1 \\
Canada}
\email{cstewart@uwaterloo.ca}

\author{Stanley Yao Xiao}
\address{Mathematical Institute \\
University of Oxford \\
Oxford \\  OX2 6GG  \\
United Kingdom}
\email{stanley.xiao@maths.ox.ac.uk}

\indent
\thanks{The research of the first author was supported in part by the Canada Research Chairs Program and by Grant A3528 from the Natural Sciences and Engineering Research Council of Canada.}

\date{\today}


\begin{abstract} Let $F$ be a binary form with integer coefficients, non-zero discriminant and degree $d$ with $d$ at least $3$ and let $r$ denote the largest degree of an irreducible factor of $F$ over the rationals. Let $k$ be an integer with $k \geq 2$ and suppose that there is no prime $p$ such that $p^k$ divides $F(a,b)$ for all pairs of integers $(a,b)$. Let $R_{F,k}(Z)$ denote the number of $k$-free integers of absolute value at most $Z$ which are represented by $F$. We prove that there is a positive number $C_{F,k}$ such that $R_{F,k}(Z)$ is asymptotic to $C_{F,k} Z^{\frac{2}{d}}$ provided that $k$ exceeds $  \frac{7r}{18}$ or $(k,r)$ is $(2,6)$ or $(3,8)$. 
\end{abstract}

\maketitle

\section{Introduction}
\label{Intro}

Let $F$ be a binary form with integer coefficients, non-zero discriminant $\Delta(F)$ and degree $d$ with $d \geq 3$. For any positive number $Z$ let $\R_F(Z)$ denote the set of non-zero integers $h$ with $|h| \leq Z$ for which there exist integers $x$ and $y$ such that $F(x,y) = h$. Denote the cardinality of a set $\S$ by $|\S|$ and let $R_F(Z) = |\R_F(Z)|$. In \cite{SX} Stewart and Xiao proved that there exists a positive number $C_F$ such that
\begin{equation} \label{RFZ} R_F(Z) \sim C_F Z^{\frac{2}{d}}. \end{equation}
Such a result had been obtained earlier by Hooley in \cite{Hoo0}, \cite{Hoo1-2}, \cite{Hoo2} and \cite{Hoo22} when $F$ is an irreducible binary cubic form, when $F$ is a quartic form of the shape
\[F(x,y) = ax^4 + bx^2 y^2 + cy^4.\]
and when $F$ is the product of linear forms with integer coefficients. In addition, a number of authors including Bennett, Dummigan, and Wooley \cite{BDW}, Browning \cite{Br2}, Greaves \cite{Grea}, Heath-Brown \cite{HB3}, Hooley \cite{Hoo1-1}, \cite{Hool1}, \cite{Hool2}, Skinner and Wooley \cite{SWoo} and Wooley \cite{Woo} obtained asymptotic estimates for $R_F(Z)$ when $F$ is a binomial form. \\

Let $k$ be an integer with $k \geq 2$. An integer is said to be $k$-free if it is not divisible by the $k$-th power of a prime number. For any positive number $Z$ let $\R_{F,k}(Z)$ denote the set of $k$-free integers $h$ with $|h| \leq Z$ for which there exist integers $x$ and $y$ such that $F(x,y) = h$ and put $R_{F,k}(Z) = |\R_{F,k}(Z)|$. Extending work of Hooley \cite{Hoo0}, \cite{Hooley}, Gouv\^{e}a and Mazur \cite{GM} in 1991 proved that if there is no prime $p$ such that $p^2$ divides $F(a,b)$ for all pairs of integers $(a,b)$, if all the irreducible factors of $F$ over $\bQ$ have degree at most $3$ and if $\ep$ is a positive real number then there are positive numbers $C_1$ and $C_2$, which depend on $\ep$ and $F$, such that if $Z$ exceeds $C_1$ then
\begin{equation} \label{basic inequality} R_{F,2}(Z) > C_2 Z^{\frac{2}{d}}.\end{equation}
This was subsequently extended by Stewart and Top in \cite{ST2}. Let $r$ be the largest degree of an irreducible factor of $F$ over $\bQ$. Let $k$ be an integer with $k \geq 2$ and suppose that there is no prime $p$ such that $p^k$ divides $F(a,b)$ for all integer pairs $(a,b)$. They showed, by utilizing work of Greaves \cite{Gre} and Erd\H{o}s and Mahler \cite{EM},
that if $k$ is at least $(r-1)/2$ or $k = 2$ and $r = 6$ then there are positive numbers $C_3$ and $C_4$, which depend on $k$ and $F$, such that if $Z$ exceeds $C_3$ then
\begin{equation} \label{basic inequality 2} R_{F,k}(Z) > C_4 Z^{\frac{2}{d}}. \end{equation}
The estimates (\ref{basic inequality}) and (\ref{basic inequality 2}) were used by Gouv\^{e}a and Mazur \cite{GM} and Stewart and Top \cite{ST2} in order to estimate, for any elliptic curve defined over $\bQ$, the number of twists of the curve for which the rank of the Mordell-Weil group is at least $2$. \\

For any real number $x$ let $\lceil x \rceil$ denote the least integer $u$ such that $x \leq u$. In 2016 \cite{X} Xiao extended the range for which (\ref{basic inequality 2}) holds by generalizing the determinant method of Heath-Brown \cite{HB1} and Salberger \cite{S1}, \cite{S2} to the setting of weighted projective space. He proved that if 
\begin{equation} \label{k inequality} k > \min \left\{ \frac{7r}{18}, \left \lceil \frac{r}{2} \right \rceil - 2 \right \}, \end{equation}
 and $(k,r)$ is not (3,8) then (\ref{basic inequality 2}) holds. In addition, the related problem of estimating $B_{F,k}(Z)$, the number of pairs of integers $(x,y)$ with $\max\{|x|, |y|\} \leq Z$ for which $F(x,y)$ is $k$-free, has been studied by Browning \cite{B2},  Filaseta, \cite{Fil}, Granville \cite{Gran}, Greaves \cite{Gre}, Helfgott \cite{Hel}, Hooley \cite{Hoo3}, \cite{Hoo4},    Murty and Pasten \cite{MP}, Poonen \cite{Poo} and Xiao \cite{X}.  Recently Bhargava \cite{Bha} and Bhargava, Shankar and Wang \cite{BSW} have extended these estimates to the case of discriminant forms.  \\

By building on the method used to prove (\ref{RFZ}) we are now able to give an asymptotic estimate for $R_{F,k}(Z)$ provided that $k$ satisfies (\ref{k inequality}). Such an estimate has not previously been established for any integer $k \geq2$ and any binary form $F$ with integer coefficients, degree at least $3$ and non-zero discriminant. We are able to prove the following result.

\begin{theorem} \label{MT1} Let $F$ be a binary form with integer coefficients, non-zero discriminant and degree $d$ with $d \geq 3$ and let $r$ denote the largest degree of an irreducible factor of $F$ over the rationals. Let $k$ be an integer with $k \geq 2$ and suppose that there is no prime $p$ such that $p^k$ divides $F(a,b)$ for all pairs of integers $(a,b)$. Suppose that (\ref{k inequality}) holds. Then there exists a positive number $C_{F,k}$ such that
\begin{equation} \label{est1} R_{F,k}(Z) = C_{F,k} Z^{\frac{2}{d}} + O_{F,k} \left(Z^{\frac{2}{d}} /g_{k,r}( Z ) \right) \end{equation}
where
\begin{equation} \label{gkr} g_{k,r}(Z) = \begin{cases} \log Z \log \log Z & \text{if } (k,r) \ne (2,6) \text{ or } (3,8) \\ \\
(\log Z)^{\frac{(d-2)(0.7043)}{d}}     & \text{if } (k,r) = (2,6) \\ \\
\left(\dfrac{\log \log Z}{\log \log \log Z} \right)^{1 - \frac{2}{d}} & \text{if } (k,r) = (3,8). 
\end{cases} \end{equation}
\end{theorem}

Throughout this article we make use of the standard notation "$O$", "$o$" and "$\sim$", for instance as in Section 1.6 of \cite{HW}, with the convention that the implicit constant denoted by the symbol "$O$" may be determined in terms of the subscripts attached to it. \\

For a positive number $Z$ we put
\[\N_F(Z) = \{(x,y) \in \bZ^2 : 1 \leq |F(x,y)| \leq Z\}\]
and
\[N_F(Z) = |\N_F(Z)|.\]
We also put
\begin{equation} \label{area} A_F = \mu (\{(x,y) \in \bR^2 : |F(x,y)| \leq 1 \})\end{equation}
where $\mu(\cdot)$ denotes the area of a set in $\bR^2$. In 1933 Mahler \cite{Mah} proved that if $F$ is a binary form with integer coefficients and degree $d$ with $d \geq 3$ which is irreducible over $\bQ$ then
\[N_F(Z) = A_F Z^{\frac{2}{d}} + O_F \left(Z^{\frac{1}{d-1}} \right).\]
The assumption that $F$ is irreducible may be replaced with the weaker requirement that $F$ have non-zero discriminant; see \cite{Thun}. \\

Let $k$ be an integer with $k \geq 2$. For a positive number $Z$ we put
\[\N_{F,k}(Z) = \{(x,y) \in \bZ^2 : F(x,y) \text{ is }k\text{-free and } 1 \leq |F(x,y)| \leq Z\}\]
and
\[N_{F,k}(Z) = |\N_{F,k}(Z)|.\]
For each positive integer $m$ we put
\[\rho_F(m) = |\{(i,j) \in \{0, \cdots, m-1\}^2 : F(i,j) \equiv 0 \pmod{m} \}|\]
and
\[\lambda_{F,k} = \prod_p \left(1 - \frac{\rho_F(p^k)}{p^{2k}} \right),\]
where the product is taken over the primes $p$. Observe that the product converges since $k \geq 2$ and $\rho_F(p^k)$ is at most $p^{2k-2} + d p^k$ provided that $p$ does not divide the discriminant $\Delta(F)$, see \cite{ST1}. Further $\lambda_{F,k} = 0$ whenever there is a prime $p$ such that $p^k$ divides $F(a,b)$ for all $(a,b)$ in $\bZ^2$. Next we put 
\begin{equation} \label{main constant 1} c_{F,k} = \lambda_{F,k} A_F .\end{equation}
In order to establish Theorem \ref{MT1} we require the following extension of Mahler's result.
\begin{theorem} \label{MT2} Let $F$ be a binary form with integer coefficients, non-zero discriminant and degree $d$ with $d \geq 3$ and let $r$ denote the largest degree of an irreducible factor of $F$ over $\bQ$.  Let $k$ be an integer with $k \geq 2$ and suppose that  (\ref{k inequality}) holds. Then, with $c_{F,k}$ defined by (\ref{main constant 1}), we have
\begin{equation} \label{est11}N_{F,k}(Z) = c_{F,k} Z^{\frac{2}{d}} + O_{F,k} \left(Z^{\frac{2}{d}} /g_{k,r} (Z) \right) \end{equation}
with $g_{k,r}(Z)$ given by (\ref{gkr}). 
\end{theorem}

Let $A$ be an element of $\GL_2(\bQ)$ with
\[A = \begin{pmatrix} a_1 & a_2 \\ a_3 & a_4 \end{pmatrix}.\]
Put $F_A(x,y) = F(a_1 x + a_2 y, a_3 x + a_4 y)$. We say that $A$ fixes $F$ if $F_A = F$. The set of $A$ in $\GL_2(\bQ)$ which fix $F$ is the \emph{automorphism group} of $F$ and we shall denote it by $\Aut F$. Let $G_1$ and $G_2$ be subgroups of $\GL_2(\bQ)$. We say that they are equivalent under conjugation if there is an element $T$ in $\GL_2(\bQ)$ such that $G_1 = T G_2 T^{-1}$. There are 10 equivalence classes of finite subgroups of $\GL_2(\bQ)$ under $\GL_2(\bQ)$-conjugation to which $\Aut F$ might belong, see \cite{New} and \cite{SX}, and we give a representative of each equivalence class together with its generators in Table 1 below. \\
\begin{center} \label{Table 1}
\begin{tabular}
{ |p{1.5cm}|p{4.5cm}|p{1.5cm}|p{4.5cm}|  }
 \hline
 \multicolumn{4}{|c|}{Table 1} \\
 \hline
 Group & Generators & Group & Generators \\
 \hline
 & & & \\
 $\CC_1$   & $\begin{pmatrix} 1 & 0 \\ 0 & 1 \end{pmatrix}$    & $\DD_1$ &   $\begin{pmatrix} 0 & 1 \\ 1 & 0 \end{pmatrix}$ \\ & & & \\
 $\CC_2$ &   $\begin{pmatrix} -1 & 0 \\ 0 & -1 \end{pmatrix}$  & $\DD_2$ & $\begin{pmatrix} 0 & 1 \\ 1 & 0 \end{pmatrix}, \begin{pmatrix} -1 & 0 \\ 0 & -1 \end{pmatrix}$ \\ & & & \\
 $\CC_3$ & $\begin{pmatrix} 0 & 1 \\ -1 & -1 \end{pmatrix}$ & $\DD_3$ & $\begin{pmatrix} 0 & 1 \\ 1 & 0 \end{pmatrix}, \begin{pmatrix} 0 & 1 \\ -1 & -1 \end{pmatrix}$\\ & & & \\ 
 $\CC_4$    & $\begin{pmatrix} 0 & 1 \\ -1 & 0 \end{pmatrix}$ & $\DD_4$ &  $\begin{pmatrix} 0 & 1 \\ 1 & 0 \end{pmatrix}, \begin{pmatrix} 0 & 1 \\ -1 & 0 \end{pmatrix}$ \\ & & & \\
 $\CC_6$ &  $\begin{pmatrix} 0 & - 1 \\ 1 & 1 \end{pmatrix}$  & $\DD_6$ & $\begin{pmatrix} 0 & 1 \\ 1 & 0 \end{pmatrix}, \begin{pmatrix} 0 & 1 \\ -1 & 1 \end{pmatrix}$ \\ & & & \\
  \hline
\end{tabular}
\end{center}

Let $\Lambda$ be the sublattice of $\bZ^2$ consisting of $(u,v)$ in $\bZ^2$ for which $A \binom{u}{v}$ is in $\bZ^2$ for all $A$ in $\Aut F$. \\

When $\Aut F$ is conjugate to $\DD_3$ it has three subgroups $G_1, G_2$ and $G_3$ of order $2$ with generators $A_1, A_2$ and $A_3$ respectively, and one, $G_4$ say, of order $3$ with generator $A_4$. Let $\Lambda_i = \Lambda(A_i)$ be the sublattice of $\bZ^2$ consisting of $(u,v)$ in $\bZ^2$ for which $A_i \binom{u}{v}$ is in $\bZ^2$ for $i=1,2,3,4$. \\

When $\Aut F$ is conjugate to $\DD_4$ there are three subgroups $G_1, G_2$ and $G_3$ of order $2$ of $\Aut F / \{\pm I\}$ where $I$ denotes the $2 \times 2$ identity matrix. Let $\Lambda_i$ be the sublattice of $\bZ^2$ consisting of $(u,v)$ in $\bZ^2$ for which $A \binom{u}{v}$ is in $\bZ^2$ for $A$ in a generator of $G_i$ for $i=1,2,3$. \\

Finally when $\Aut F$ is conjugate to $\DD_6$ there are three subgroups $G_1, G_2$ and $G_3$ of order $2$ and one, $G_4$ say, of order $3$ in $\Aut F /\{\pm I\}$. Let $A_i$ be in a generator of $G_i$ for $i = 1,2,3,4$. Let $\Lambda_i$ be the sublattice of $\bZ^2$ consisting of $(u,v)$ in $\bZ^2$ for which $A_i \binom{u}{v}$ is in $\bZ^2$ for $i=1,2,3,4$. \\

Let $L$ be a sublattice of $\bZ^2$. We define $c_{F,k,L}$ in the following manner. For any basis $\{\omega_1, \omega_2\}$ of $L$ with $\omega_1 = (a_1, a_3)$ and $\omega_2 = (a_2, a_4)$ we define $F_{\omega_1, \omega_2}(x,y) = F(a_1 x + a_2 y, a_3 x + a_4 y)$. Notice that if $\{\omega_1', \omega_2'\}$ is another basis for $L$ then it is related to $\{\omega_1, \omega_2\}$ by a unimodular transformation. As a consequence, 
\[c_{F_{\omega_1, \omega_2}, k} = c_{F_{\omega_1', \omega_2'}, k}\]
and so we may define $c_{F,k,L}$ by putting
\[c_{F,k,L} = c_{F_{\omega_1, \omega_2}, k}.\]
Observe that if $L = \bZ^2$ then $c_{F,k,L} = c_{F,k}$. For brevity, we shall write 
\begin{equation} \label{CL} c(L) = c_{F,k,L}. \end{equation}

We are now able to determine the positive number $C_{F,k}$ in (\ref{est1}) of Theorem \ref{MT1} explicitly in terms of $\Aut F$ and the lattices described above. 

\begin{theorem} \label{MT3} The positive number $C_{F,k}$ in the statement of Theorem \ref{MT1} is given by the following table: 
\begin{center} \label{Table 2}
\begin{tabular}
{ |p{1.3cm}|p{3.1cm}|p{1.3cm}|p{9.6cm}|  }
 
 \hline
$\Rep (F)$  & $C_{F,k}$ & $\Rep (F)$ & $C_{F,k}$ \\
 \hline
 $\CC_1$   & $c_{F,k}$    & $\DD_1$ &   $c_{F,k} - \dfrac{c(\Lambda)}{2}$ \\ & & & \\
 $\CC_2$ &   $\dfrac{c_{F,k}}{2}$  & $\DD_2$ & $\dfrac{1}{2} \left(c_{F,k} - \dfrac{c(\Lambda)}{2} \right)$ \\ & & & \\
 $\CC_3$ & $c_{F,k} - \dfrac{2 c(\Lambda)}{3}$ & $\DD_3$ & $c_{F,k} - \dfrac{ c(\Lambda_1)}{2} - \dfrac{c(\Lambda_2)}{2} - \dfrac{c(\Lambda_3)}{2} - \dfrac{2 c(\Lambda_4)}{3} + \dfrac{4 c(\Lambda)}{3} $\\ & & & \\
 $\CC_4$    & $\dfrac{1}{2} \left(c_{F,k} - \dfrac{c(\Lambda)}{2} \right)$ & $\DD_4$ &  $\dfrac{1}{2} \left(c_{F,k} - \dfrac{ c(\Lambda_1)}{2} - \dfrac{c(\Lambda_2)}{2} - \dfrac{ c(\Lambda_3)}{2} + \dfrac{3 c(\Lambda)}{4} \right)$ \\ & & & \\
 $\CC_6$ &  $\dfrac{1}{2} \left(c_{F,k} - \dfrac{2 c(\Lambda)}{3} \right)$  & $\DD_6$ & $\dfrac{1}{2} \left(c_{F,k} - \dfrac{ c(\Lambda_1)}{2} - \dfrac{ c(\Lambda_2)}{2} - \dfrac{ c(\Lambda_3)}{2} - \dfrac{2c(\Lambda_4)}{3} + \dfrac{4 c(\Lambda)}{3} \right)$ \\ & & & \\
  \hline 
\end{tabular}
\end{center}
Here $\Rep(F)$ denotes a representative of the equivalence class of $\Aut F$ under $\GL_2(\bQ)$-conjugation, $\Lambda$ and $\Lambda_i$'s are defined above, $c_{F,k}$ is as in (\ref{main constant 1}), and $c(\Lambda)$ and $c(\Lambda_i)$ as in (\ref{CL}).

\end{theorem}

We conjecture that the estimates for $R_{F,k}(Z)$ in Theorem \ref{MT1} and for $N_{F,k}(Z)$ in Theorem \ref{MT2} apply without hypothesis (\ref{k inequality}).

\begin{conjecture} \label{MC} Let $F$ be a binary form with integer coefficients, non-zero discriminant and degree $d$ with d at least $3$. Let $k$ be an integer larger than $1$. Then either $c_{F,k} = 0$ or
\begin{equation} \label{estMC1} N_{F,k}(Z) \sim c_{F,k}Z^{\frac{2}{d}} \end{equation}
where $c_{F,k}$ is defined by (\ref{main constant 1}). If there is no prime $p$ such that $p^k$ divides $F(a,b)$ for all pairs of integers $(a,b)$ then
\begin{equation} \label{estMC2} R_{F,k}(Z) \sim C_{F,k}Z^{\frac{2}{d}} \end{equation}
where $C_{F,k}$ is the positive number given by Theorem \ref{MT3}.
\end{conjecture}

Let $F$ be a binary form with integer coefficients, non-zero discriminant and degree $d$ with $d\geq 3$. Granville \cite{Gran} established an asymptotic estimate for $B_{F,2}(Z)$, the number of pairs of integers $(x,y)$ with absolute value at most $Z$ for which $F(x,y)$ is squarefree subject to the $abc$ conjecture, see eg. \cite{SY}. Let $k$ be an integer with $k>1$. The same analysis allows one to give an asymptotic estimate for $B_{F,k}(Z)$, the number of pairs of integers $(x,y)$ with absolute value at most $Z$ for which $F(x,y)$ is $k$-free. We may use such an estimate in conjunction with the arguments used to prove Theorem \ref{MT1} and Theorem \ref{MT2} in order to prove Conjecture \ref{MC}, see the final paragraph of Section 6. In particular, Conjecture  \ref{MC} follows from the $abc$ conjecture.
\section{Preliminary lemmas}
\label{prelims}

Let $F(x,y)$ be a binary form with integer coefficients, non-zero discriminant and degree $d$ with $d \geq 3$.  Suppose that $F$ factors over $\bC$ as 
\begin{equation} \label{factorization} F(x,y) = \prod_{i=1}^d (\gamma_i x + \beta_i y) \end{equation}
and put
\[\HH(F) = \prod_{i=1}^d \sqrt{|\gamma_i|^2 + |\beta_i|^2}. \]
Then $\HH(F)$ does not depend on the factorization in (\ref{factorization}). \\

A special case of Theorem 3 of Thunder \cite{Thun} is the following explicit version of a result of Mahler \cite{Mah}.

\begin{lemma} \label{Thunder lemma} Let $F$ be a binary form with integer coefficients, non-zero discriminant and degree $d \geq 3$. Then, with $A_F$ defined by (\ref{area}), we have
\[ \left \lvert N_F(Z) - A_F Z^{\frac{2}{d}} \right \rvert = O\left(Z^{\frac{1}{d-1}} \HH(F)^{d-2} \right),\]
where the implied constant is absolute. 
\end{lemma}

We may write
\begin{equation} \label{fact1} F(x,y) = a\prod_{i=1}^d ( x - \alpha_i y) \end{equation}
where $a$ is a positive integer and $\alpha_1, \dots, \alpha_d$ are the roots of $F(x,1)$ provided that $y$ is not a factor of $F(x,y)$. In the latter case, since the discriminant of $F$ is non-zero, we have
\begin{equation} \label{fact2} F(x,y) = ay\prod_{i=1}^{d-1} ( x - \alpha_i y). \end{equation}
Put
\begin{equation} \label{EF} E_F = \frac{2 \max\limits_{1\leq j\leq k}(1,|\alpha_j|)}{\min(1,\min\limits_{i\neq j}|\alpha_i-\alpha_j|)} \end{equation}
where $k=d$ if (\ref{fact1}) holds and $k=d-1$ if (\ref{fact2}) holds.
\begin{lemma} \label{cusp lemma} Let $F$ be a binary form with integer coefficients, non-zero discriminant and degree $d$ with $d \geq 3$. Let $Z$ be a positive real number. For any positive real number $\beta$ larger than $E_F$ the number of pairs of integers $(x,y)$ with
\[0 < |F(x,y)| \leq Z\]
for which
\[\max\{|x|, |y|\} > Z^{\frac{1}{d}} \beta \]
is 
\[O_F \left(Z^{\frac{1}{d}} \log Z + \frac{Z^{\frac{2}{d}}}{\beta^{d-2}} \right).\]
\end{lemma}

\begin{proof} We shall follow Heath-Brown's proof of Theorem 8 in \cite{HB1}. Accordingly put 
\[S(Z;C) = | \{ (x,y) \in \bZ^2 : 0 < |F(x,y)| \leq Z, C < \max \{|x|, |y|\} \leq 2C, \gcd(x,y) = 1 \}|.\]
Suppose that
\begin{equation} \label{C1}C \geq Z^{\frac{1}{d}} E_F.  \end{equation}
Heath-Brown observes that by Roth's theorem $S(Z;C) = 0$ unless $C \ll_F Z^2$ .\\

Suppose that we are in the case when (\ref{fact1}) holds and that $(x,y)$ is a pair of integers with $\gcd(x,y)=1$,
\[0 < |F(x,y)| \leq Z\]
and
\begin{equation} \label{C2}C < \max (|x|, |y|) \leq 2C \end{equation} 
Further suppose that $i_0$ is an index for which
\[ |x-\alpha_{i_0}y| \leq |x-\alpha_jy|\]
for $j=1,\dots,d.$ Note that then
\begin{equation} \label{eq3} |x-\alpha_{i_0}y| \leq Z^{1/d}. \end{equation}

We have two cases to consider. The first case is when $\max(|x|,|y|) =|y|$. In this case we have, for $j\ne i_0,$
\begin{equation} \label{eq4} |x-\alpha_jy| = |(x-\alpha_{i_0}y) + (\alpha_{i_0}-\alpha_j)y| \geq |\alpha_{i_0}-\alpha_j||y| - |x-\alpha_{i_0}y| \end{equation}
and, by (\ref{C1}), (\ref{C2}) and (\ref{eq3}),
\begin{equation} \label{eq5} \frac{1}{2}|\alpha_{i_0}- \alpha_j||y| - |x-\alpha_{i_0}y| \geq \frac{1}{2}|\alpha_{i_0}- \alpha_j|Z^{1/d}E_F-Z^{1/d} \geq 0. \end{equation}
Thus, by (\ref{eq4}) and (\ref{eq5}),
\begin{equation} \label{eq6} |x-\alpha_jy| \geq  \frac{1}{2}|\alpha_{i_0}- \alpha_j||y| \geq \frac{1}{2}|\alpha_{i_0}- \alpha_j|C. \end{equation}

The second case is when $\max (|x|,|y|) = |x|.$ Then
\begin{equation} \label{eq7} |\alpha_{i_0}(x-\alpha_jy)| = |((\alpha_{i_0} - \alpha_j)x+ \alpha_j(x-\alpha_{i_0}y)| \geq |\alpha_{i_0}- \alpha_j||x| - |\alpha_j||x - \alpha_{i_0}y|, \end{equation}
and, by (\ref{EF}), (\ref{eq6}) and (\ref{eq7}),
\begin{equation} \label{eq8} \frac{1}{2}|\alpha_{i_0}- \alpha_j||x| - |\alpha_j||x-\alpha_{i_0}y| \geq \frac{1}{2}|\alpha_{i_0}- \alpha_j|Z^{1/d}E_F-|\alpha_j|Z^{1/d} \geq 0. \end{equation}
Thus, by (\ref{eq7}) and (\ref{eq8}),
\begin{equation} \label{eq9} |x-\alpha_jy| \geq  \frac{1}{2|\alpha_{i_0}|}|\alpha_{i_0}- \alpha_j|C. \end{equation}

It now follows from (\ref{C2}), (\ref{eq6}) and (\ref{eq9}) that
\begin{equation} \label{eq10} C \ll_F |x-\alpha_jy| \ll_F C. \end{equation}
We obtain (\ref{eq10}) in a similar fashion when (\ref{fact2}) holds.

Thus, by (\ref{eq10}),
\begin{equation} \label{eq11}  |x-\alpha_{i_0}y| \ll_F Z/C^{d-1}. \end{equation}
The number of coprime integer pairs $(x,y)$ satisfying (\ref{C2}) and (\ref{eq11}) for some index $i_0$ is an upper bound for $S(Z;C)$ and therefore, by Lemma 1, part (vii) of \cite{HB1},
\begin{equation} \label{SZ bound} S(Z;C) \ll_F 1 + \frac{Z}{C^{d-2}}. \end{equation}
Put
\[S^{(1)}(Z;C) = | \{(x,y) \in \bZ^2 : 0 < |F(x,y)| \leq Z, C < \max\{|x|, |y|\}, \gcd(x,y) = 1\}|.\]
Therefore on replacing $C$ by $2^j C$ in (\ref{SZ bound}) for $j = 1, 2, \cdots$ and summing we find that
\[S^{(1)}(Z; C) \ll_F \log Z + \frac{Z}{C^{d-2}}. \]
Next put
\[S^{(2)}(Z;C) = |\{(x,y) \in \bZ^2 : 0 < |F(x,y)| \leq Z, C < \max\{|x|, |y|\} \}|.\]
Then
\[S^{(2)}(Z;C) \ll_F \sum_{h \leq Z^{1/d}} S^{(1)} \left(\frac{Z}{h^d}, \frac{C}{h} \right)\]
and since $C > Z^{\frac{1}{d}} E_F$ we see that
\[\frac{C}{h} > \left(\frac{Z}{h^d}\right)^{\frac{1}{d}} E_F,\]
hence
\begin{align*} S^{(2)}(Z; C) & \ll_F \sum_{h \leq Z^{1/d}} \left(\log Z + \frac{Z}{h^2 C^{d-2}} \right) \\
& \ll_F Z^{\frac{1}{d}} \log Z + \frac{Z}{C^{d-2}}. \end{align*} 
Our result now follows on taking $C = Z^{\frac{1}{d}} \beta$ since $\beta > E_F$. 
\end{proof}
For any positive real numbers $Z$ and $\beta$ put
\[N_F(Z, \beta) = |\{(x,y) \in \bZ^2 :  |F(x,y)| \leq Z, \max\{|x|, |y|\} \leq Z^{\frac{1}{d}} \beta \}|.\]

\begin{lemma} \label{small primes} Let $F$ be a binary form of degree $d \geq 3$ with integer coefficients and non-zero discriminant. Let $E_F$ be as in (\ref{EF}) and suppose that $\beta$ is a real number with $\beta > E_F$. Then
\[N_F(Z, \beta) = A_F Z^{\frac{2}{d}} + O_F\left(Z^{\frac{1}{d-1}} + Z^{\frac{1}{d}} \beta + Z^{\frac{2}{d}} \beta^{-(d-2)} \right).\]
\end{lemma}

\begin{proof} This follows from Lemma \ref{Thunder lemma} and Lemma \ref{cusp lemma} on noting that the number of pairs of integers $(x,y)$ with $\max\{|x|, |y|\} \leq Z^{\frac{1}{d}} \beta$ for which $F(x,y) = 0$ is at most $O_F\left(Z^{\frac{1}{d}} \beta \right)$. 
\end{proof}

Next we introduce the quantity
\[\tilde{N}_F(Z, \beta) = |\{(x,y) \in \bZ^2 : |F(x + \theta_1, y + \theta_2)| \leq Z, \]

\[0 \leq \theta_1 < 1, 0 \leq \theta_2 < 1, \max\{|x + \theta_1|, |y + \theta_2| \} \leq Z^{\frac{1}{d}} \beta \}|\]
in order to facilitate the determination of the main terms in Theorems \ref{MT1} and \ref{MT2}.

\begin{lemma} \label{n-tilde} Let $F$ be a binary form with integer coefficients, non-zero discriminant and degree $d \geq 3$. Let $E_F$ be as in (\ref{EF}) and suppose that $\beta$ is a real number with $\beta > E_F$. Then
\[\tilde{N}_F(Z, \beta) = A_F Z^{\frac{2}{d}} + O_F \left(Z^{\frac{1}{d-1}} + Z^{\frac{2}{d}} \beta^{-(d-2)} + Z^{\frac{1}{d}} \beta^{d-1} \right).\]
\end{lemma}

\begin{proof} Note that there is a positive number $\kappa$, which depends on $F$, such that for $(\theta_1, \theta_2) \in \bR^2$ with $|\theta_i| \leq 1$ for $i = 1,2$ we have 
\[|F(x + \theta_1, y + \theta_2)| \leq |F(x,y)| + \kappa \max \{|x|, |y|\}^{d-1}.\]
Put
\begin{equation} \label{Z1} Z_1 = Z - \kappa \left(Z^{\frac{1}{d}} \beta \right)^{d-1}. \end{equation}
Thus if $\max\{|x|, |y|\} \leq Z^{\frac{1}{d}} \beta$ and 
\[|F(x,y)| \leq Z_1\]
then, for $(\theta_1, \theta_2) \in \bR^2$ with $|\theta_i| \leq 1$ for $i = 1,2$, we have
\begin{equation} \label{near Z} |F(x + \theta_1, y + \theta_2)| \leq Z. \end{equation}
Observe that
\begin{equation} \label{Z1 near Z} Z_1^{\frac{2}{d}} = \left(Z - \kappa \left(Z^{\frac{1}{d}} \beta \right)^{d-1} \right)^{\frac{2}{d}} = Z^{\frac{2}{d}} + O_F \left(Z^{\frac{1}{d}} \beta^{d-1} \right).\end{equation}
Thus, since $Z_1 \leq Z$, by Lemma \ref{small primes} the number of pairs of integers $(x,y)$ for which $\max\{|x|, |y|\} \leq Z_1^{\frac{1}{d}} \beta$ and (\ref{near Z}) holds is at least
\begin{equation} \label{lower bound} A_F Z_1^{\frac{2}{d}} + O_F \left(Z_1^{\frac{1}{d-1}} + Z_1^{\frac{1}{d}} \beta + Z_1^{\frac{2}{d}} \beta^{-(d-2)} \right), \end{equation}
and so by (\ref{Z1 near Z}) and (\ref{lower bound}), at least
\begin{equation} \label{lower bound Z} A_F Z^{\frac{2}{d}} + O_F \left(Z^{\frac{1}{d-1}} + Z^{\frac{2}{d}} \beta^{-(d-2)} + Z^{\frac{1}{d}} \beta^{d-1} \right). \end{equation}
Therefore, by Lemma \ref{small primes} and (\ref{lower bound Z}), 
\[\tilde{N}_F(Z, \beta) = A_F Z^{\frac{2}{d}} + O_F \left(Z^{\frac{1}{d-1}} + Z^{\frac{2}{d}} \beta^{- (d-2)} + Z^{\frac{1}{d}} \beta^{d-1} \right),\]
as required. 
\end{proof} 

We now put, for a real number $Z$, an integer $k$ with $k \geq 2$ and positive numbers $\gamma$ and $\beta$, 
\[N_{F,k}(Z, \gamma, \beta) = |\{(x,y) \in \bZ^2 : |F(x,y)| \leq Z, \max\{|x|, |y|\} \leq Z^{\frac{1}{d}} \beta \text{ and } \]
\[ F(x,y) \text{ is not divisible by } p^k \text{ for any prime } p \text{ with } p \leq \gamma \}|,\]
and
\[N_{F,k}(Z, \gamma) = |\{(x,y) \in \bZ^2 : 0 < |F(x,y)| \leq Z \text{ and } F(x,y) \text{ is not divisible by } \]
\[p^k \text{ for any prime } p \text{ with } p \leq \gamma \}|.\]

\begin{lemma} \label{partial NF} Let $F$ be a binary form with integer coefficients, non-zero discriminant and degree $d$ with $d\geq 3$. Then
\[N_{F,k}\left(Z, \frac{1}{2kd} \log Z \right) = c_{F,k} Z^{\frac{2}{d}} + O_{F,k} \left(Z^{\frac{2}{d}} / (\log Z \log \log Z) \right)\]
with $c_{F,k}$ given by (\ref{main constant 1}).
\end{lemma}

\begin{proof} We have
\begin{align*} N_{F,k} \left(Z, \frac{1}{2kd} \log Z  \right) & = N_{F,k} \left(Z, \frac{1}{2kd} \log Z, (\log Z)^6 \right)  + O_{F,k} \left(Z^{\frac{1}{d}} (\log Z)^6 \right)  \\
& + O_{F,k} \left(|\{(x,y) \in \bZ^2 : 0 < |F(x,y)| \leq Z \text{ and } \max\{|x|, |y|\} > Z^{\frac{1}{d}} (\log Z)^6 \}| \right).
\end{align*}
By Lemma \ref{cusp lemma}, since $ d \geq 3$,
\begin{equation} \label{small NF} N_{F,k}\left(Z, \frac{1}{2kd} \log Z \right) = N_{F,k} \left(Z, \frac{1}{2kd} \log Z, (\log Z)^6 \right) + O_{F,k} \left(Z^{\frac{2}{d}}/ (\log Z)^6 \right). \end{equation}
Next we put
\[V = V(d,k,Z) = \prod_{p \leq \log Z / (2kd)} p^k,\]
where the product is taken over primes $p$. For each pair of integers $(a,b)$ we define $B(a,b)$ by 
\[B(a,b) = \{(x,y) \in \bR^2 : aV \leq x < (a+1) V, bV \leq y < (b+1) V \}.\]
Observe that $B(a,b)$ is a square in $\bR^2$. We say that $B(a,b)$ is admissible if 
\begin{equation} \label{admissible} |F(x,y)| \leq Z \text{ and } \max\{|x|, |y|\} \leq Z^{\frac{1}{d}} (\log Z)^6 \end{equation} 
whenever $(x,y)$ is in $B(a,b)$. Let $B_1$ denote the number of admissible squares $B(a,b)$. We have
\[B_1 = \tilde{N}_F \left(\frac{Z}{V^d}, (\log Z)^6 \right)\]
and so by Lemma \ref{n-tilde}, since $d \geq 3$, 
\begin{equation} \label{B1} B_1 = A_F \frac{Z^{\frac{2}{d}}}{V^2} + O_F \left( \left(\frac{Z}{V^d} \right)^{\frac{1}{d-1}} + \frac{Z^{\frac{2}{d}}}{V^2 (\log Z)^6} + \frac{Z^{\frac{1}{d}}}{V} (\log Z)^{6(d-1)} \right) \end{equation}
Therefore by Lemma \ref{small primes} with $\beta = (\log Z)^6$ the number of pairs of integers $(x,y)$ for which (\ref{admissible}) holds and which are not in an admissible square is 
\begin{equation} \label{not square} O_F \left(Z^{\frac{1}{d-1}} V^{\frac{d-2}{d-1}} + Z^{\frac{2}{d}} /(\log Z)^{6} + Z^{\frac{1}{d}} V (\log Z)^{6(d-1)} \right). \end{equation}
We may now apply the Chinese Remainder Theorem to conclude that within each admissible square the number of integer pairs $(x,y)$ for which $F(x,y)$ is not divisible by $p^k$ for any prime $p$ with $p \leq \frac{1}{2kd} \log Z$ is precisely 
\[\prod_{p \leq \log Z/(2kd)} \left(1 - \frac{\rho_F(p^k)}{p^{2k}} \right) V^2. \]
Thus the number of integer pairs $(x,y)$ in an admissible square with $F(x,y)$ not divisible by $p^k$ for any prime $p$ with $p \leq \frac{1}{2kd} \log Z$ is
\begin{equation} \label{good square} B_1  \prod_{p \leq \log Z/(2kd)} \left(1 - \frac{\rho_F(p^k)}{p^{2k}} \right) V^2. \end{equation}
Therefore, by (\ref{B1}), (\ref{not square}) and (\ref{good square}), 
\begin{equation} \label{small prime 2} N_{F,k}\left(Z, \frac{1}{2kd} \log Z \right) = A_F \prod_{p \leq \log Z/(2kd)} \left(1 - \frac{\rho_F(p^k)}{p^{2k}} \right) Z^{\frac{2}{d}} + \end{equation}
\[ O_{F,k} \left(Z^{\frac{1}{d-1}} V^{\frac{d-2}{d-1}} + Z^{\frac{1}{d}} V (\log Z)^{6(d-1)} + Z^{\frac{2}{d}} /(\log Z)^{6} \right).\]
By the Prime Number Theorem,
\[ V = O \left(Z^{\frac{1}{2d} + \frac{1}{d^2}} \right) \]
and so, by (\ref{small prime 2}), 
\begin{equation} \label{PNT} N_{F,k}\left(Z, \frac{1}{2kd} \log Z \right) = A_F \prod_{p \leq \log Z/(2kd)} \left(1 - \frac{\rho_F(p^k)}{p^{2k}} \right) Z^{\frac{2}{d}} + O_{F,k} \left(Z^{\frac{2}{d}}/ (\log Z)^{6} \right). \end{equation}

Note that the number of integer pairs $(a,b)$ with $0 \leq a < p^k$ and $0 \leq b < p^k$ for which $p$ divides both $a$ and $b$ is $p^{2k-2}$. Further the number of pairs $(a,b)$ for which $p$ does not divide both $a$ and $b$ and for which $F(a,b) \equiv 0 \pmod{p^k}$ is at most $dp^k$ provided that $p$ does not divide $\Delta(F)$, see \cite{ST1}. Thus for primes $p$ which do not divide $\Delta(F)$, we have
\begin{equation} \label{rho} \rho_F(p^k) \leq p^{2k-2} + dp^k \leq (d+1) p^{2k-2}, \end{equation}
since $k \geq 2$. Put
\[P = \prod_{p \leq \log Z/(2kd)} \left(1 - \frac{\rho_F(p^k)}{p^{2k}} \right),   P_1 = \prod_p \left(1 - \frac{\rho_F(p^k)}{p^{2k}} \right)\]
and
\[t = \sum_{p > \log Z/(2kd)} \log \left(1 - \frac{\rho_F(p^k)}{p^{2k}} \right).\] 
Then
\[P_1 - P = P(e^t - 1) = -P \left(-t - \frac{t^2}{2!} - \frac{t^3}{3!} - \cdots \right).\]
Since $t$ is negative,
\begin{equation} \label{partial product} 0 \leq P - P_1 \leq -Pt. \end{equation}
Further,
\[-t = O_{F,k} \left(\sum_{p > \log Z/(2kd)} \frac{\rho_F(p^k)}{p^{2k}} \right) \]
and by (\ref{rho}), 
\begin{equation} \label{t bound} -t = O_{F,k} \left(\sum_{p > \log Z/(2kd)} \frac{1}{p^2} \right). \end{equation}
We have
\[\sum_{p > \log Z/(2kd)} \frac{1}{p^2} = \sum_{j=0}^\infty \sum_{2^j \frac{\log Z}{2kd} < p < 2^{j+1} \frac{\log Z}{2kd}} \frac{1}{p^2}\]
and so, by the Prime Number Theorem,
\begin{equation} \label{PNT2} \end{equation}
\begin{align*} \sum_{p > \log Z/(2kd)} \frac{1}{p^2} & = O_k \left(\sum_{j=0}^\infty \left(\frac{2^{j+1} \log Z}{(j+1) \log 2 + \log \log Z} \right) \frac{1}{2^{2j} (\log Z)^2} \right) \\
& = O_k \left(\frac{1}{\log Z \log \log Z} \right).
\end{align*}
Therefore, by (\ref{partial product}), (\ref{t bound}) and (\ref{PNT2}), 
\begin{equation} \label{P estimate} P = P_1 + O_{F,k} \left(\frac{1}{\log Z \log \log Z} \right).\end{equation}
It now follows from (\ref{PNT}) and (\ref{P estimate}) that
\begin{equation} \label{small prime 3} N_{F,k}\left(Z, \frac{1}{2kd} \log Z\right) = c_{F,k} Z^{\frac{2}{d}} + O_{F,k} \left(Z^{\frac{2}{d}} /(\log Z \log \log Z) \right), \end{equation}
as required. \end{proof}

We say that an integer $h$ is \emph{essentially represented} by $F$ if whenever $(x_1, y_1), (x_2, y_2)$ are in $\bZ^2$ and
\[F(x_1, y_1) = F(x_2, y_2) = h\]
then there exists $A$ in $\Aut F$ such that
\[A \binom{x_1}{y_1} = \binom{x_2}{y_2}.\]
We remark that if there is only one integer pair $(x,y)$ for which $F(x,y) = h$ then $h$ is essentially represented since $I$ is in $\Aut F$. \\

For any positive number $Z$ let $R_F^{(2)} (Z)$ denote the number of integers $h$ with $0 < |h| \leq Z$ which are not essentially represented by $F$. For each binary form $F$ with integer coefficients, non-zero discriminant and degree $d$ with $d\geq 3$ we define $\beta_F$ in the following way. If $F$ has a linear factor in $\bR [x,y]$ we put

\begin{equation} \label{Beta D} \beta_F = \begin{cases} \dfrac{12}{19} & \text{if } d = 3 \text{ and } F \text{ is irreducible over $\bQ$} \\ \\
\dfrac{4}{7} & \text{if } d = 3 \text{ and } F \text{ has exactly one linear factor over $\bQ$} \\ \\
\dfrac{5}{9} & \text{if } d = 3 \text{ and } F \text{ has three linear factors over $\bQ$} \\ \\

\dfrac{3}{(d-2)\sqrt{d} + 3} & \text{if } 4 \leq d \leq 8 \\ \\ 
\dfrac{1}{d-1} & \text{if } d \geq 9. \end{cases}
\end{equation}
If $F$ does not have a linear factor over $\bR$ then $d$ is even and we put

\begin{equation} \label{Beta DD} \beta_F = \begin{cases} \dfrac{3}{8} & \text{if } d = 4 \\ \\
\dfrac{1}{2\sqrt{6} } & \text{if } d=6\\ \\ 
\dfrac{1}{d-1} & \text{if } d \geq 8. \end{cases}
\end{equation}

In \cite{SX}, Stewart and Xiao, building on work of Heath-Brown \cite{HB1}, Salberger \cite{S1}, \cite{S2} and Colliot-Th\'el\`ene \cite{HB1}, proved the following result. 

\begin{lemma} \label{SX non-est} Let $F$ be a binary form with integer coefficients, non-zero discriminant and degree $d$ with $d \geq 3$. Then for each $\ep > 0$, 
\[R_F^{(2)}(Z)  = O_{F,\ep} \left(Z^{\beta_F + \ep}\right)\]
where $\beta_F$ is given by (\ref{Beta D}) and (\ref{Beta DD}). 
\end{lemma}

The proof of Lemma \ref{SX non-est} is based on the $p$-adic determinant method of Heath-Brown as elaborated in \cite{HB1}. \\

Recall that if $F$ is a binary form we denote by $\Lambda$ the sublattice of $\bZ^2$ consisting of integer pairs $(u,v)$ for which $A \binom{u}{v}$ is in $\bZ^2$ for all $A$ in $\Aut F$. Further, if $\Aut F$ is conjugate to $\DD_3, \DD_4$ and $\DD_6$ we define $\Lambda_i$ for $i = 1,2,3,4$ as in our discussion following Table 1 in the introduction. 

\begin{lemma} \label{common core} Let $F$ be a binary form with integer coefficients, non-zero discriminant and degree $d \geq 3$. If $A$ is an element of order $3$ in $\Aut F$ then
\[\Lambda(A^2) = \Lambda(A).\]
If $\Aut F$ is equivalent under conjugation in $\GL_2(\bQ)$ to $\DD_3, \DD_4$ or $\DD_6$ then
\[\Lambda_i \cap \Lambda_j = \Lambda \text{ for } i \ne j.\]
\end{lemma}

Lemma \ref{common core} is Lemma 3.2 of \cite{SX} .
\section{Outline of the proof of Theorem \ref{MT2}}

Let $N_1$ denote the number of integer pairs $(x,y)$ for which
\begin{itemize}
\item[(i)] $0 < |F(x,y)| \leq Z$, 

and

\item[(ii)] $p^k \nmid F(x,y)$ for $1 \leq p \leq \frac{1}{2kd} \log Z$. 
\end{itemize}
By Lemma \ref{partial NF}
\begin{equation} \label{N1} N_1 = c_{F,k} Z^{\frac{2}{d}} + O_{F,k} \left(Z^{\frac{2}{d}} / (\log Z \log \log Z) \right). \end{equation}

Our objective is to show that the number $N$ of integer pairs for which (i) holds and 
\begin{itemize}
\item[(iii)] $p^k \nmid F(x,y)$ for $p$ a prime,
\end{itemize}
satisfies a similar estimate to (\ref{N1}). 
To that end let $N_2$ denote the number of integer pairs $(x,y)$ for which (i) holds and $p$ divides both $x$ and $y$ for some prime $p > \frac{1}{2kd} \log Z$. Let $N_3$ denote the number of pairs of integers $(x,y)$ for which (i) holds and for some prime $p$ with
\[\frac{1}{2kd} \log Z < p \leq (\log Z)^9\]
we have $p^k | F(x,y)$ and $p \nmid \gcd(x,y)$. Let $N_4$ denote the number of integer pairs $(x,y)$ for which (i) holds and for some prime $p$ with
\[(\log Z)^9 < p \leq \frac{Z^{\frac{2}{d}}}{(\log Z)^9},\]
$p^k | F(x,y)$ and $p \nmid \gcd(x,y)$. Finally let $N_5$ denote the number of integer pairs $(x,y)$ for which (i) holds and for some prime $p$ with 
\[\frac{Z^{\frac{2}{d}}}{(\log Z)^9} < p,\]
$p^k | F(x,y)$ and $p \nmid \gcd(x,y)$. Then
\begin{equation} \label{big small decomp} N = N_1 + O (N_2 + N_3 + N_4 + N_5).
\end{equation}
In order to establish Theorem \ref{MT2} it suffices, by (\ref{N1}) and (\ref{big small decomp}), to prove that
\[N_i = O_{F,k} \left(Z^{\frac{2}{d}} /u(z)\right)\]
for $i = 2,3,4$ and $5$ where
\begin{equation} \label{T1} u(z) = \log Z\log \log Z \end{equation}
when $k$ and $r$ satisfy (\ref{k inequality}) with $(k,r)$ not $(2,6)$ or $(3,8)$,
\begin{equation} \label{T2} u(z) = (\log Z)^{\frac{(d-2)\delta}{d}} \end{equation}
with $\delta=0.7043$ when $(k,r)$ is $(2,6)$ and 
\begin{equation} \label{T3} u(z) = (\log \log Z / \log \log \log Z)^{1-{\frac{2}{d}}}  \end{equation} when $(k,r)$ is $(3,8)$. \\

We may suppose that $F$ factors over $\bQ$ as
\begin{equation} \label{T5} F(x,y) = \prod_{i=1}^t F_i(x,y)\end{equation}
with $F_i$ in $\bZ[x,y]$ and irreducible over $\bQ$ for $i = 1, \cdots, t$. Let $r_i$ be the degree of $F_i$ for $i = 1, \cdots, t$ and put
\begin{equation} \label{largest degree} r = \max\{r_1, \cdots, r_t\}. \end{equation}
\section{An estimate for $N_2$ and for $N_3$}

Notice that if $p$ divides $a$ and $b$ and $0 < |F(a,b)| \leq Z$ then $|F(a,b)| = p^d |F(a/p, b/p)|$, so $p \leq Z^{\frac{1}{d}}$. As a consequence 
\begin{equation} \label{N2 est} N_2 = O \left(\sum_{\frac{1}{2kd} \log Z < p \leq Z^{\frac{1}{d}}} \left \lvert  \left \{(x,y) \in \bZ^2 : 0 < |F(x,y)| \leq \frac{Z}{p^d} \right \} \right \rvert \right).
\end{equation}
Further, by Lemma \ref{Thunder lemma}, for each prime $p$ with $p \leq Z^{\frac{1}{d}}$, 
\begin{equation} \label{N2 est 2} \left \lvert \left \{ (x,y) \in \bZ^2 : 0 < |F(x,y)| \leq \frac{Z}{p^d} \right \} \right \rvert = A_F \frac{Z^{\frac{2}{d}}}{p^2} + O_F \left( \left(\frac{Z}{p^d}\right)^{\frac{1}{d-1}} \right). 
\end{equation}
Thus by (\ref{N2 est}) and (\ref{N2 est 2}), 
\begin{equation} \label{N2 est 3} N_2 = O_F \left( \left(\sum_{\frac{1}{2kd} \log Z < p } \frac{1}{p^2} \right) Z^{\frac{2}{d}} \right). \end{equation}
It now follows from (\ref{PNT2}) and (\ref{N2 est 3}) that
\begin{equation} \label{N2 est 4} N_2 = O_{F,k} \left(Z^{\frac{2}{d}} / (\log Z \log \log Z) \right). \end{equation}

The integer pairs $(a,b)$ with $F(a,b) \equiv 0 \pmod{p^k}$ and for which $p$ does not divide both $a$ and $b$ lie in at most $d$ sublattices $L_\theta$ of $\bZ^2$, provided that $p$ does not divide the discriminant $\Delta(F)$ of $F$, see \cite{Gre}. Each sublattice $L_\theta$ is defined by a congruence of the form
\[a \equiv \theta b \pmod{p^k}\]
for some integer $\theta$ with $0 \leq \theta < p^k$. Let $(a_1, a_3)$ and $(a_2, a_4)$ be a basis for $L_\theta$ chosen so that $\max\{|a_1|, |a_2|, |a_3|, |a_4|\}$ is minimized. Then
\[\max\{|a_1|, |a_2|, |a_3|, |a_4| \} \leq p^k.\]
Put
\[F_{L_\theta}(x,y) = F(a_1 x + a_2 y, a_3 x + a_4 y)\]
and notice that
\[\left \lvert \N_F(Z) \cap L_\theta \right \rvert = N_{F_{L_\theta}} (Z).\]
Observe that
\[\HH(F_{L_\theta}) \leq 4^dp^{kd} \HH(F).\]
Therefore by Lemma \ref{Thunder lemma}
\[N_{F_{L_\theta}}(Z) = A_{F_{L_\theta}} Z^{\frac{2}{d}} + O_F \left(p^{kd} Z^{\frac{1}{d-1}} \right)\]
and, since the lattice $L_\theta$ has determinant $p^k$, 
\begin{equation} \label{NFL} N_{F_{L_\theta}} (Z) = \frac{A_F Z^{\frac{2}{d}}}{p^k} + O_F \left(p^{kd} Z^{\frac{1}{d-1}} \right). \end{equation} 
Thus
\begin{align*} N_3 & = O_{F,k} \left(Z^{\frac{2}{d}} \sum_{\frac{1}{2kd} \log Z < p \leq (\log Z)^9} \frac{1}{p^k} \right) \\
& = O_{F,k} \left(Z^{\frac{2}{d}} \sum_{\frac{1}{2kd} \log Z < p} \frac{1}{p^2} \right), 
\end{align*}
and so, by (\ref{PNT2}),
\begin{equation} \label{N3 est} N_3 = O_{F,k} \left(Z^{\frac{2}{d}} / (\log Z \log \log Z) \right). \end{equation} 

\section{An estimate for $N_4$}

In order to estimate $N_4$ we note that
\begin{equation} \label{N4 est 1} N_4 = O \left(N_{4}^{(1)} + N_{4}^{(2)} \right) \end{equation}
where $N_4^{(1)}$ is the number of integer pairs $(x,y)$ for which
\begin{equation} \label{N4 xy range} \max\{|x|, |y|\} \leq Z^{\frac{1}{d}} (\log Z)^{7/2} \end{equation}
and for which
 $p^k$ divides $F(x,y)$ for some $p$ with
\begin{equation} \label{N4 p range} (\log Z)^9 < p \leq Z^{\frac{2}{d}} /(\log Z)^{9} \end{equation}
which does not divide both $x$ and $y$.
Further $N_4^{(2)}$ is the number of integer pairs $(x,y)$ for which $0 < |F(x,y)| \leq Z$ and
\[\max\{|x|, |y|\} > Z^{\frac{1}{d}} (\log Z)^{7/2}.\] 
By Lemma \ref{cusp lemma} we have, since $d$ is at least 3,
\begin{equation} \label{N4b est} N_4^{(2)} = O_F \left(Z^{\frac{2}{d}} /(\log Z)^{\frac{7}{2}} \right). \end{equation}
It remains only to estimate $N_4^{(1)}$ and we shall do so by a modification of an argument of Greaves \cite{Gre} based on the geometry of numbers. \\

Recall (\ref{T5}). For $i=1,\dots,t$ we let $N_{4,i}^{(1)}$ be the number of integer pairs $(x,y)$ for which $F(x,y) \neq 0$, (\ref{N4 xy range}) holds and $p^k$ divides $F_{i}(a,b)$ for some prime satisfying (\ref{N4 p range}) which does not divide both $x$ and $y$. Then
\begin{equation} \label{N44 est} N_{4}^{(1)} = O(N_{4,i}^{(1)} + \dots +  N_{4,t}^{(1)}).    \end{equation}

Suppose that $(a,b)$ is an integer pair for which $p^k$ divides $F_i(a,b)$ for some prime $p$ satisfying (\ref{N4 p range}) which does not divide both $a$ and $b$. We may suppose that $Z$ is sufficiently large that $(\log Z)^9$ exceeds $|\Delta(F)|$. Then $(a,b)$ belongs to one of at most $r_i$ lattices $L_\theta$ defined by a congruence
\[a \equiv \theta b \pmod{p^k}.\]
Following Greaves \cite{Gre} we let $M = M(\theta, p^k)$ denote the minimal possible positive value of $\max\{|a|, |b|\}$ as we range over $(a,b)$ in $L_\theta$. For any real number $X$ let $N_{\theta,k}(X)$ denote the number of pairs $(a,b)$ in $L_\theta$ for which $|a| \leq X$ and $|b| \leq X$. Then, by Lemma 1 of \cite{Gre},
\begin{equation} \label{Greaves eqn} N_{\theta,k}(X) \leq \frac{4X^2}{p^k} + O \left(\frac{X}{M}\right). \end{equation}
It then follows from (\ref{Greaves eqn}) with $X = Z^{\frac{1}{d}} (\log Z)^{\frac{7}{2}} $ that 
\begin{align*} N_{4,i}^{(1)} & \leq \sum_{(\log Z)^9< p \leq Z^{\frac{2}{d}} (\log Z)^{-9}} \sum_\theta N_{\theta, k}(X)\\
& \leq \sum_{(\log Z)^9 < p \leq Z^{\frac{2}{d}} (\log Z)^{-9}} \sum_\theta \left(\frac{4Z^{\frac{2}{d}} (\log Z)^7}{p^k} + O \left( \frac{Z^{\frac{1}{d}} (\log Z)^{\frac{7}{2}}} {M(\theta, p^k)} \right) \right). \end{align*} 
For each prime $p$ we have at most $d$ terms $\theta$ in the inner sum. Thus 
\begin{align} \label{mediumMT} N_{4,i}^{(1)} & = O_F \left(Z^{\frac{2}{d}} (\log Z)^7 \sum_{(\log Z)^9 < p \leq Z^{\frac{2}{d}} (\log Z)^{-9}} \frac{1}{p^k} \right)  \\
& + O_F \left(Z^{\frac{1}{d}} (\log Z)^{\frac{7}{2}} \sum_{(\log Z)^9 < p \leq Z^{\frac{2}{d}} (\log Z)^{-9}} \sum_\theta \frac{1}{M(\theta, p^k)} \notag \right). 
\end{align}
Certainly
\begin{align} \label{medium prime MT} \sum_{(\log Z)^9 < p \leq Z^{\frac{2}{d}} (\log Z)^{-9}} \frac{1}{p^k} & = O \left(\sum_{(\log Z)^9 < p} \frac{1}{p^k} \right) \\
& = O \left(\frac{1}{(\log Z)^{9(k-1)}} \right) \notag \\
& = O \left(\frac{1}{(\log Z)^9}\right). \notag
\end{align}
Further, since $M(\theta, p^k)$ is at least 1, 
\begin{equation} \label{short vector} \sum_{(\log Z)^9 < p \leq Z^{\frac{1}{2d}}(\log Z)^2} \sum_\theta \frac{1}{M(\theta, p^k)} = O_F \left(Z^{\frac{1}{2d}}(\log Z)^2 \right). 
\end{equation}
It remains to estimate $S$ where
\[S = \sum_{Z^{\frac{1}{2d}}(\log Z)^2  < p < Z^{\frac{2}{d}} (\log Z)^{-9}} \sum_\theta \frac{1}{M(\theta, p^k)}. \]
Notice that if $r_i=1$ then $F_i(x,y)=O_F(Z^{\frac{1}{d}}(\log Z)^{\frac{7}{2}})$ and so if $p^k$ divides $F_i(x,y)$ then, since $k\geq 2$, $p=O_F(Z^{\frac{1}{2d}}( \log Z)^{\frac{7}{4}})$.Thus if $r_i=1$ then
\begin{equation} \label{S44} S=O_F(1). \end{equation}

We shall now estimate $S$ under the assumption that $r_i>1$.
We put $S = S_1 + S_2$ where $S_1$ is the sum over pairs $p, \theta$ with 
\[M(\theta, p^k) \geq \frac{Z^{\frac{1}{d}} }{(\log Z)^5}\]
and $S_2$ is the sum over the other pairs $(p, \theta)$. Certainly 
\begin{align} \label{S1} S_1 & = O_F \left(\sum_{p \leq Z^{\frac{2}{d}} (\log Z)^{-9}} \frac{(\log Z)^5}{Z^{\frac{1}{d}}} \right) \\
& = O_F \left(\frac{Z^{\frac{1}{d}}}{(\log Z)^5} \right). \notag
\end{align}
On the other hand $S_2$ consists of the sum over pairs $p, \theta$ with 
\[1 \leq M \leq \frac{Z^{\frac{1}{d}}}{(\log Z)^5},\]
and $p > Z^{\frac{1}{2d}}(\log Z)^2$. To each pair $p, \theta$ we may associate a pair of integers $(r,s)$  for which $\max\{|r|, |s|\} = M(\theta, p^k)$. Note that since $r_i>1$ we have $F_i(r,s)\neq 0$. Further there are at most $O_F(1)$ pairs $(p,\theta)$ with $p> Z^{\frac{1}{2d}}(\log Z)^2$ which can be associated with a given pair $(r,s)$ since $F_i(r,s)=O_F(Z^{\frac{r_i}{d}})$. Thus
\begin{align} \label{S2} S_2 & = O \left(\sum_{1 \leq s \leq Z^{\frac{1}{d}} (\log Z)^{-5}} \frac{1}{s} \sum_{0 \leq r \leq s} \sum_{\substack{p^k | F(r,s) \\ F(r,s) \ne 0 \\ p > Z^{\frac{1}{2d}}(\log Z)^2}} 1 \right)  \\
& = O_{F,k} \left( Z^{\frac{1}{d}} (\log Z)^{-5} \right). \notag
\end{align} 
 Therefore, by (\ref{N44 est}), (\ref{mediumMT}), (\ref{medium prime MT}), (\ref{short vector}), (\ref{S44}), (\ref{S1}) and (\ref{S2}), 
\begin{equation} \label{N4a est} N_4^{(1)} = O_{F,k} \left(Z^{\frac{2}{d}} / (\log Z)^{\frac{3}{2}} \right). \end{equation}
Further, by (\ref{N4 est 1}), (\ref{N4b est}) and (\ref{N4a est}), 
\begin{equation} \label{N4 final} N_4 = O_{F,k}\left(Z^{\frac{2}{d}} /(\log Z)^{\frac{3}{2}} \right). \end{equation}

\section{An estimate for $N_5$}

For any real number $T$ let $B^*_{F,k}(T)$ denote the number of pairs of integers $(x,y)$ with $\max (|x|,|y|) \leq T$ and for which $F(x,y)$ is divisible by $p^k$ with $p$ a prime larger than $T^2/(\log T)^{12}$. Then
\begin{equation} B^*_{F,k}(T) = O(B^*_{F_1,k}(T)+...+   B^*_{F_t,k}(T)). \end{equation}
If $r \leq 2k+1$ then Greaves used Selberg's sieve to prove that
\begin{equation} \label{Greaves bound 1} B^*_{F_i,k}(T) = O_{F,k} \left(T^{2 - \frac{1}{20}} \right) \end{equation}
for $i = 1, \cdots, t$. This follows from the proof of Lemma 4 of \cite{Gre} on taking $x = T$ and $\eta = (\log T)^{-16}$; Greaves required the constraint $\eta \geq (\log T)^{-2}$ but it may be replaced with the weaker constraint $\eta \geq (\log T)^{-16}$. 
Xiao dealt with the case when 
\[\frac{7}{18} < \frac{k}{r} < \frac{1}{2}\]
in \cite{X} by means of the determinant method applied to weighted projective spaces. It follows from \cite{X} that in this case
\begin{equation} \label{big prime bound 1} B^*_{F_i,k}(T) = O_{F,k} \left(T^{\frac{2}{d}} /(\log T)^{4} \right) \end{equation}
for $i = 1, \cdots, t$.
Therefore for $\frac{k}{r} > \frac{7}{18}$
\begin{equation} \label{big prime bound 11} B^*_{F,k}(T) = O_{F,k} \left(T^{2}/ (\log T)^{4} \right). \end{equation}

By a result of Helfgott, see the proof of Theorem 5.2 of \cite{Hel}, when $(k,r)$ is $(2,6)$
\begin{equation} \label{big prime bound 12} B^*_{F,2}(T) = O_{F,2} \left(T^{2}/ (\log T)^{\delta} \right) \end{equation}
where
\[\delta = 0.7043.\]

Hooley in 2009 established an asymptotic estimate for the number of integer pairs $(x,y)$ in a box for which $F(x,y)$ is cubefree when $F$ is a binary form of degree $8$ with integer coefficients which is irreducible over the rationals, see \cite{Hoo3} and Theorem 2 of \cite{Hoo4}. Xiao \cite{XCJM} extended this work to decomposable forms $F$ and an examination of his proof yields an explicit error term from which we find that
\begin{equation} \label{N5665} B^*_{F,3}(T) = O_F \left(T^{2}/(\log \log T / \log \log \log T) \right) \end{equation} 
when $(k,r)$ is $(3,8)$. \\

Define $g(T)$ by
\begin{equation} \label{Beta DDD} g(T) = \begin{cases} (\log T)^4 & \text{if } \frac{k}{r} > \frac{7}{18} \\ \\
(\log T)^{\delta} & \text{if } (k,r)=(2,6)\\ \\ 
\log \log T / \log \log \log T & \text{if } (k,r) = (3,8). \end{cases}
\end{equation}
Then by (\ref{big prime bound 11}), (\ref{big prime bound 12}) and (\ref{N5665}),
\begin{equation} \label{big prime bound 13} B^*_{F,k}(T) = O_{F,k} \left(T^{2}/g(T) \right) \end{equation}
for $(k,r)$ satisfying (\ref{k inequality}).

 Put
\begin{equation} \label{B33} f(T) = g\left(T^{\frac{1}{d}}\right)^{\frac{1}{d}}. \end{equation}
Let $N_5^{(1)}$ be the number of integer pairs $(x,y)$ for which $F(x,y) \ne 0$ and $p^k | F(x,y)$ for some prime $p$ with
\begin{equation} \label{large prime} p > Z^{\frac{2}{d}}/(\log Z)^{9}\end{equation}
which does not divide both $x$ and $y$ and for which
\begin{equation} \label{xy bound 2} \max\{|x|, |y|\} \leq Z^{\frac{1}{d}}f(Z). \end{equation} 
Further, write $N_5^{(2)}$ for the number of integer pairs $(x,y)$ for which $0 < |F(x,y)| \leq Z$ and 
\begin{equation} \label{xy bound 3} \max\{|x|, |y|\} > Z^{\frac{1}{d}}f(Z). \end{equation} 
Notice that $N_5 = O \left(N_5^{(1)} + N_5^{(2)}\right)$. By Lemma \ref{cusp lemma},
\begin{equation} \label{N56} N_5^{(2)} = O_{F,k} \left(Z^{\frac{2}{d}} /f(Z)^{d-2} \right) = O_{F,k} \left(Z^{\frac{2}{d}}/g(Z^{\frac{1}{d}})^{\frac{d-2}{d}}  \right). \end{equation} Furthermore, on taking $T = Z^{\frac{1}{d}}f(Z)$, we see from (\ref{big prime bound 13}) and (\ref{B33}) that
\[N_5^{(1)} = O_{F,k} \left(Z^{\frac{2}{d}} f(Z)^2/g(Z^{\frac{1}{d}}f(Z)) \right)\]
Since $g(T)$ is eventually increasing and tends to infinity with $T$ it follows from (\ref{B33}) that $f(Z)$ is at least $1$ for $Z$ sufficiently large. We then have $g(Z^{\frac{1}{d}}) \leq g(Z^{\frac{1}{d}}f(Z))$ and so
\[N_5^{(1)} = O_{F,k} \left(Z^{\frac{2}{d}} f(Z)^2/g(Z^{\frac{1}{d}}) \right)\]But $f(Z)^2/g(Z^{\frac{1}{d}}) = f(Z)^{-d+2}$, by (\ref{B33}), and thus
\[N_5^{(1)} = O_{F,k} \left(Z^{\frac{2}{d}}/g(Z^{\frac{1}{d}})^{\frac{d-2}{d}} \right).\]
Therefore
\begin{equation} \label{N566} N_5 = O_{F,k} \left(Z^{\frac{2}{d}}/g(Z^{\frac{1}{d}})^{\frac{d-2}{d}} \right). \end{equation} 
Theorem \ref{MT2} now follows from (\ref{N1}), (\ref{big small decomp}), (\ref{N2 est 4}), (\ref{N3 est}), (\ref{N4 final}), (\ref{Beta DDD}) and (\ref{N566}).  \\

If $F$ is a binary form with integer coefficients, nonzero discriminant and degree at least $3$ and $k$ is an integer larger than $1$ then there exists a positive monotone increasing function  $g_1(t)$ on the positive real numbers with $0 \leq g_1(t) \leq \log (t+2)$ for all positive real numbers $t$ and
\[ \lim_{t\to\infty} g_1(t) = \infty \]
such that
\begin{equation} \label{big prime bound 14} B^*_{F,k}(T) = O_{F,k} \left(T^{2}/g_1(T) \right), \end{equation}
subject to the $abc$ conjecture. Granville \cite{Gran} showed this when $k=2$ and his argument extends readily to the general case. Arguing as above we deduce that Conjecture \ref{MC} holds for $N_{F,k}(Z)$. With this estimate for $N_{F,k}(Z)$ we are then able to establish Conjecture \ref{MC} for $R_{F,k}(Z)$ as in the next section. 
\section{The proof of Theorems \ref{MT1} and \ref{MT3}}

If $\Aut F = \CC_1$ then every integer pair $(x,y)$ for which $F(x,y)$ is essentially represented with $0 < |F(x,y)| \leq Z$ gives rise to a distinct integer $h$ with $0 < |h| \leq Z$. It follows from Theorem \ref{MT2} and Lemma \ref{SX non-est} that
\begin{equation} \label{RFK main} R_{F,k}(Z) = c_{F,k} Z^{\frac{2}{d}} + O_{F,k} \left(Z^{\frac{2}{d}} / u(z)\right)\end{equation}
where $u(z)$ is defined as in (\ref{T1}) when $k$ and $r$ satisfy (\ref{k inequality}) with $(k,r)$ not $(2,6)$ or $(3,8)$, as in (\ref{T2}) when $(k,r)$ is $(2,6)$ and satisfies (\ref{T3})  when $(k,r)$ is $(3,8)$. Similarly if $\Aut F = \CC_2$ then (\ref{RFK main}) holds with $\frac{1}{2} c_{F,k}$ in place of $c_{F,k}$. \\

Suppose now that $\Aut F$ is conjugate to $\CC_3$. Then for $A$ in $\Aut F$ with $A \ne I$ we have, by Lemma \ref{common core}, $\Lambda(A) = \Lambda(A^2)$. Thus whenever $(x,y)$ is in $\Lambda$, $F(x,y) = h$ and $h$ is essentially represented there are exactly two other pairs $(x_1, y_1), (x_2, y_2)$ for which $F(x_i, y_i) = h$ for $i = 1,2$. When $(x,y)$ is in $\bZ^2$ but not in $\Lambda$ and $F(x,y)$ is essentially represented then $F(x,y)$ has exactly one representative. \\

Let $\{\omega_1, \omega_2\}$ be a basis for $\Lambda$ with $\omega_1 = (a_1, a_3)$ and $\omega_2 = (a_2, a_4)$ and such that $\max\{|a_1|,|a_2|,|a_3|,|a_4\}$ is minimized. Recall that
\[F_{\omega_1, \omega_2}(x,y) = F(a_1 x + a_2 y, a_3 x + a_4 y).\]
Since
\[|\N_{F,k}(Z) \cap \Lambda | = N_{F_{\omega_1, \omega_2}, k} (Z),\]
by Theorem \ref{MT2} we have
\begin{equation} \label{NFK L} |\N_{F,k}(Z) \cap \Lambda | = c(\Lambda) Z^{\frac{2}{d}} + O_{F, k} \left(Z^{\frac{2}{d}} / u(z)\right). \end{equation}
Note that since $\omega_1$ and $\omega_2$ are chosen so that $\max\{|a_1|,|a_2|,|a_3|,|a_4|\}$ is minimized the implicit constant in the error term may be determined in terms of $F$ and $k$. 
By (\ref{NFK L}) and  Lemma \ref{SX non-est} the number of integer pairs $(x,y)$ in $\Lambda$ for which $0 < |F(x,y)| \leq Z$ and $F(x,y)$ is $k$-free and essentially represented is 
\[c(\Lambda) Z^{\frac{2}{d}} + O_{F, k} \left(Z^{\frac{2}{d}} / u(z) \right).\]
Each pair $(x,y)$ is associated with two other pairs which represent the same integer. These pairs yield
\begin{equation} \label{3-peat} \frac{c(\Lambda)}{3} Z^{\frac{2}{d}} + O_{F,k} \left(Z^{\frac{2}{d}} / u(z)\right) \end{equation}
integers $h$ with $0 < |h| \leq Z$. It now follows from Theorem \ref{MT2} and Lemma \ref{SX non-est} that there are
\begin{equation} \label{C3 1-peat} (c_{F,k} - c(\Lambda)) Z^{\frac{2}{d}} + O_{F,k} \left(Z^{\frac{2}{d}} / u(z)\right) \end{equation}
integer pairs $(x,y)$ not in $\Lambda$ for which $F(x,y)$ is $k$-free and essentially represented and each pair gives rise to an integer $h$ with $0 < |h| \leq Z$ which is uniquely represented by $F$. It follows from (\ref{3-peat}) and (\ref{C3 1-peat}) that when $\Aut F$ is equivalent to $\CC_3$ we have
\[R_{F,k}(Z) = \left(c_{F,k} - \frac{2}{3} c(\Lambda) \right) Z^{\frac{2}{d}} + O_{F,k} \left(Z^{\frac{2}{d}} / u(z) \right).\]
A similar analysis applies to the case when $\Aut F$ is equivalent to $\DD_1, \DD_2, \CC_4$ or $\CC_6$. These groups are cyclic with the exception of $\DD_2$ but $\DD_2/\{\pm I\}$ is cyclic and that is sufficient for our purposes. \\

We are left with the possibility that $\Aut F$ is conjugate to $\DD_3, \DD_4$ or $\DD_6$. We first consider the case when $\Aut F$ is equivalent to $\DD_4$. In this case (\ref{NFK L}) holds as before and since each $h$ which is essentially represented by $F$ and for which $h=F(x,y)$ with $(x,y)$ in $\Lambda$ is represented by $8$ integer pairs the number of $k$-free integers $h$ with $0 < |h| \leq Z$ for which there exists an integer pair $(x,y)$ in $\Lambda$ with $F(x,y) = h$ is
\begin{equation} \label{D4 core} \frac{c(\Lambda)}{8} Z^{\frac{2}{d}} + O_{F,k} \left(Z^{\frac{2}{d}} / u(z) \right). \end{equation}
By Lemma \ref{common core} $\Lambda_i \cap \Lambda_j = \Lambda$ for $1 \leq i < j \leq 3$ and so the number of integer pairs $(x,y)$ in $\Lambda_1, \Lambda_2$ or $\Lambda_3$ but not in $\Lambda$ for which $F(x,y)$ is essentially represented and $k$-free with $0 < |F(x,y)| \leq Z$ is, by Theorem \ref{MT2}, 
\[\left(c(\Lambda_1) + c(\Lambda_2) + c(\Lambda_3) - 3 c (\Lambda) \right) Z^{\frac{2}{d}} + O_{F,k} \left(Z^{\frac{2}{d}} / u(z)\right).\]
Each such integer $F(x,y)$ has precisely four representatives and so the terms in $\Lambda_1, \Lambda_2, \Lambda_3$ but not in $\Lambda$ contribute 
\begin{equation} \label{D4 not core} \frac{1}{4} \left(c(\Lambda_1) + c(\Lambda_2) + c(\Lambda_3) - 3 c (\Lambda) \right) Z^{\frac{2}{d}} + O_{F,k} \left(Z^{\frac{2}{d}} / u(z)\right) \end{equation}
terms to $\R_{F,k}(Z)$. Finally the terms $(x,y)$ in $\N_{F,k}(Z)$ but not in $\Lambda_1, \Lambda_2$ or $\Lambda_3$ for which $F(x,y)$ is essentially represented have cardinality
\[\left(c_{F,k} - c(\Lambda_1) - c(\Lambda_2) - c(\Lambda_3) + 2 c(\Lambda)\right) Z^{\frac{2}{d}} + O_{F,k}\left(Z^{\frac{2}{d}} / u(z)\right). \]
Each integer represented by such a term has $2$ representations and therefore these terms contribute
\begin{equation} \label{D4 no lat} \frac{1}{2} \left(c_{F,k} - c(\Lambda_1) - c(\Lambda_2) - c(\Lambda_3) + 2 c(\Lambda) \right)Z^{\frac{2}{d}} + O_{F,k} \left(Z^{\frac{2}{d}} / u(z)\right) \end{equation}
terms to $\R_{F,k}(Z)$. It now follows from (\ref{D4 core}), (\ref{D4 not core}), (\ref{D4 no lat}) and Lemma \ref{SX non-est} that 
\[R_{F,k}(Z) = \frac{1}{2} \left(c_{F,k} - \frac{c(\Lambda_1)}{2} - \frac{c(\Lambda_2)}{2} - \frac{c(\Lambda_3)}{2} + \frac{3 c(\Lambda)}{4} \right) Z^{\frac{2}{d}} + O_{F,k}\left(Z^{\frac{2}{d}} / u(z)\right),\]
as required. \\

Next suppose that $\Aut F$ is conjugate to $\DD_3$. As before the pairs $(x,y)$ in $\N_{F,k}(Z) \cap \Lambda$ for which $F(x,y)$ is essentially represented yield
\begin{equation} \label{D3 core} \frac{c(\Lambda)}{6} Z^{\frac{2}{d}} + O_{F,k}\left(Z^{\frac{2}{d}} / u(z) \right) \end{equation} 
terms in $\R_{F,k}(Z)$. Since $\Lambda_i \cap \Lambda_j = \Lambda$ for $1 \leq i < j \leq 3$ by Lemma \ref{common core} the pairs $(x,y)$ in 
$\N_{F,k}(Z) \cap \Lambda_i \text{ for } i = 1,2,3$ which are not in $\Lambda$ and which are essentially represented contribute
\begin{equation} \label{D3 not core} \left(\frac{c(\Lambda_1)}{2} + \frac{c(\Lambda_2)}{2} + \frac{c(\Lambda_3)}{2} - \frac{3 c(\Lambda)}{2} \right) Z^{\frac{2}{d}} + O_{F,k}\left(Z^{\frac{2}{d}} / u(z) \right) \end{equation}
terms to $\R_{F,k}(Z)$. The pairs $(x,y)$ in $\N_{F,k}(Z) \cap \Lambda_4$ which are not in $\Lambda$ and for which $F(x,y)$ is essentially represented contribute
\begin{equation} \label{D3 L4} \left(\frac{c(\Lambda_4)}{3} - \frac{c(\Lambda)}{3} \right) Z^{\frac{2}{d}} + O_{F,k} \left(Z^{\frac{2}{d}} / u(z)\right) \end{equation}
terms to $\R_{F,k}(Z)$. Finally the pairs $(x,y)$ in $\N_{F,k}(Z)$ which do not lie in $\Lambda_i$ for $i = 1,2,3,4$ contribute, by Lemma \ref{common core}, 
\begin{equation} \label{D3 no lat} \left(c_{F,k} - c(\Lambda_1) - c(\Lambda_2) - c(\Lambda_3) - c(\Lambda_4)+ 3 c(\Lambda)\right) Z^{\frac{2}{d}} + O_{F,k}\left(Z^{\frac{2}{d}} / u(z) \right) \end{equation}
terms to $\R_{F,k}(Z)$. It then follows from (\ref{D3 core}), (\ref{D3 not core}), (\ref{D3 L4}), (\ref{D3 no lat}) and Lemma \ref{SX non-est} that
\[R_{F,k}(Z) = \left(c_{F,k} - \frac{c(\Lambda_1)}{2} - \frac{c(\Lambda_2)}{2} - \frac{c(\Lambda_3)}{2} - \frac{2 c(\Lambda_4)}{3} + \frac{4 c(\Lambda)}{3} \right) Z^{\frac{2}{d}} + O_{F,k} \left(Z^{\frac{2}{d}} / u(z)\right),\]
as required. \\

When $\Aut F$ is equivalent to $\DD_6$ the analysis is the same as for $\DD_3$ taking into account the fact that $\Aut F$ contains $-I$ and so the weighting factor is one half of what it is when $\Aut F$ is equivalent to $\DD_3$. \\

Finally we note that, since there is no prime $p$ such that $p^k$ divides $F(a,b)$ for all pairs of integers $(a,b)$, $c_{F,k}$ is a positive number. We have
\[ C_{F,k} \geq c_{F,k}/|\Aut F| \]
and, since the order of the automorphism group of $F$ is at most $12$, the order of $\DD_6$, we deduce that $C_{F,k}$ is positive.
This completes the proof of Theorems \ref{MT1} and \ref{MT3}.

\end{document}